\documentclass[11pt,letterpaper]{amsart}
\usepackage{cases}
\usepackage{mathrsfs}
\usepackage{color}
\usepackage{latexsym,amssymb}
\usepackage{a4wide}
\usepackage{amsthm}
\usepackage{amsmath}
\usepackage{graphicx}
\input xy
\xyoption{all}
\usepackage{verbatim}
\usepackage[lite,abbrev,msc-links]{amsrefs}

\numberwithin{equation}{section}
\begin{document}
	\theoremstyle{plain}
	\newtheorem{thm}{Theorem}[section]
	\newtheorem{lem}[thm]{Lemma}
	\newtheorem{cor}[thm]{Corollary}
	\newtheorem{cor*}[thm]{Corollary*}
	\newtheorem{prop}[thm]{Proposition}
	\newtheorem{prop*}[thm]{Proposition*}
	\newtheorem{conj}[thm]{Conjecture}
	%%%%%%%%%%%%%%%%%%%% Text roman %%%%%%%%%%%%%%%%%%%%%%%%%%%%%
	\theoremstyle{definition}
	\newtheorem{construction}{Construction}
	\newtheorem{notations}[thm]{Notations}
	\newtheorem{question}[thm]{Question}
	\newtheorem{prob}[thm]{Problem}
	\newtheorem{rmk}[thm]{Remark}
	\newtheorem{remarks}[thm]{Remarks}
	\newtheorem{defn}[thm]{Definition}
	\newtheorem{claim}[thm]{Claim}
	\newtheorem{assumption}[thm]{Assumption}
	\newtheorem{assumptions}[thm]{Assumptions}
	\newtheorem{properties}[thm]{Properties}
	\newtheorem{exmp}[thm]{Example}
	\newtheorem{comments}[thm]{Comments}
	\newtheorem{blank}[thm]{}
	\newtheorem{observation}[thm]{Observation}
	\newtheorem{defn-thm}[thm]{Definition-Theorem}
	\newtheorem*{Setting}{Setting}

	\newcommand{\sA}{\mathscr{A}}
	\newcommand{\sB}{\mathscr{B}}
	\newcommand{\sC}{\mathscr{C}}
	\newcommand{\sD}{\mathscr{D}}
	\newcommand{\sE}{\mathscr{E}}
	\newcommand{\sF}{\mathscr{F}}
	\newcommand{\sG}{\mathscr{G}}
	\newcommand{\sH}{\mathscr{H}}
	\newcommand{\sI}{\mathscr{I}}
	\newcommand{\sJ}{\mathscr{J}}
	\newcommand{\sK}{\mathscr{K}}
	\newcommand{\sL}{\mathscr{L}}
	\newcommand{\sM}{\mathscr{M}}
	\newcommand{\sN}{\mathscr{N}}
	\newcommand{\sO}{\mathscr{O}}
	\newcommand{\sP}{\mathscr{P}}
	\newcommand{\sQ}{\mathscr{Q}}
	\newcommand{\sR}{\mathscr{R}}
	\newcommand{\sS}{\mathscr{S}}
	\newcommand{\sT}{\mathscr{T}}
	\newcommand{\sU}{\mathscr{U}}
	\newcommand{\sV}{\mathscr{V}}
	\newcommand{\sW}{\mathscr{W}}
	\newcommand{\sX}{\mathscr{X}}
	\newcommand{\sY}{\mathscr{Y}}
	\newcommand{\sZ}{\mathscr{Z}}
	\newcommand{\bZ}{\mathbb{Z}}
	\newcommand{\bN}{\mathbb{N}}
	\newcommand{\bQ}{\mathbb{Q}}
	\newcommand{\bC}{\mathbb{C}}
	\newcommand{\bR}{\mathbb{R}}
	\newcommand{\bH}{\mathbb{H}}
	\newcommand{\bD}{\mathbb{D}}
	\newcommand{\bE}{\mathbb{E}}
	\newcommand{\bV}{\mathbb{V}}
	\newcommand{\cV}{\mathcal{V}}
	\newcommand{\cF}{\mathcal{F}}
	\newcommand{\bfM}{\mathbf{M}}
	\newcommand{\bfN}{\mathbf{N}}
	\newcommand{\bfX}{\mathbf{X}}
	\newcommand{\bfY}{\mathbf{Y}}
	\newcommand{\spec}{\textrm{Spec}}
	\newcommand{\dbar}{\bar{\partial}}
	\newcommand{\ddbar}{\partial\bar{\partial}}
	\newcommand{\redref}{{\color{red}ref}}
	%%%%%%%%%%%%%%%%%%%%%%%%%%%%%%%%%%%%%%%%%%%%%%%%%%%%%%%%%%%%%%
	
	\title[] {An $L^2$ Dolbeault lemma on higher direct images and its application}

	\author[Chen Zhao]{Chen Zhao}
	\email{czhao@ustc.edu.cn}
	\address{School of Mathematical Sciences,
		University of Science and Technology of China, Hefei, 230026, China}

	%%%%%%%%%%%%%%%%%%%%%%%%%%%%%%%%%%%%%%%%%%%%%%%%%%%%%%%%%%%%%%%%%%%%%%%%%%%%%%%%%%%%%%%%%%%%%%%%
	%                                          Abstract                                            %
	%%%%%%%%%%%%%%%%%%%%%%%%%%%%%%%%%%%%%%%%%%%%%%%%%%%%%%%%%%%%%%%%%%%%%%%%%%%%%%%%%%%%%%%%%%%%%%%%
	\begin{abstract}
		Given a proper holomorphic surjective morphism $f:X\rightarrow Y$ from a compact K\"ahler manifold to a compact K\"ahler manifold, and a Nakano semipositive holomorphic vector bundle $E$ on $X$, we prove Koll\'ar type vanishing theorems on cohomologies with coefficients in $R^qf_\ast(\omega_X(E))\otimes F$, where $F$ is a $k$-positive vector bundle on $Y$.
		The main inputs in the proof are the deep results on the Nakano semipositivity of the higher direct images due to Berndtsson and Mourougane-Takayama, and an $L^2$-Dolbeault resolution of the higher direct image sheaf $R^qf_\ast(\omega_X(E))$, which is of interest in itself.
	\end{abstract}
	
	\maketitle
	%\tableofcontents
	
	\section{Introduction}
	Let $f:X\rightarrow Y$ be a proper holomorphic surjective morphism from a compact K\"ahler manifold $X$ to a compact K\"ahler manifold $Y$ of dimension $m$. Let $\omega_X$ be the canonical line bundle on $X$ and let $E$ be a Nakano semipositive vector bundle on $X$. The main purpose of this article is to show the following Koll\'ar type vanishing theorem.
	\begin{thm}\label{thm_main}
		Let $F$ be a $k$-positive Hermitian vector bundle on $Y$ of rank $r$. Then
	%$\Theta_{h_F}(F)\geq_k 0$ and $\Theta_{h_F}(F)>_k 0$ in at least one point $y_0\in Y$. Then
	$$H^i(Y,R^qf_\ast(\omega_X(E))\otimes F)=0,\quad\forall i\geq 1,k\geq\min\{\dim_{\bC}Y-i+1,r\}.$$
	
	Here $\omega_X(E):=\omega_X\otimes E$ and $R^qf_\ast(-)$ denotes the $q$th higher direct image sheaf.
	\end{thm}
When $F$ is Nakano positive, it reduces to a special case of Matsumura's Koll\'ar-Ohsawa type vanishing theorem \cite{Mat2016}.
	As a corollary, we can deduce the following vanishing theorems.
	\begin{cor}\label{cor_main}
		Let $F,F_1,\dots,F_l$ be holomorphic vector bundles on $Y$ and let $L$ be a holomorphic line bundle on $Y$. Then the following hold.
		\begin{enumerate}
			\item If $F$ is ample, $L$ is nef and ${\rm rank}(F)>1$, then
			$$H^i(Y,R^qf_\ast(\omega_X(E))\otimes S^kF \otimes({\rm det} F)^2\otimes \omega_Y\otimes L)=0$$
			for any $i \geq 1$ and $k \geq \max\{m-{\rm rank}(F), 0\}$.
			\item  If $F$ is ample, $L$ is nef and ${\rm rank}(F)>1$, then
			$$H^i(Y,R^qf_\ast(\omega_X(E))\otimes F \otimes ({\rm det} F)^k \otimes \omega_Y\otimes L)=0$$
			for any $i \geq 1$ and $k \geq \max\{m+1-{\rm rank}(F), 2\}$.
			\item Let ${\rm rank}(F)> 1$. If $F$ is ample and $L$ is nef, or $F$ is nef and $L$ is ample,
			then
			$$H^i(Y,R^qf_\ast(\omega_X(E))\otimes S^mF^\ast \otimes ({\rm det} F)^t \otimes L)=0$$
			for any $i\geq1$ and $t \geq {\rm rank}(F)+m-1.$
			\item If all $F_j$ are ample and $L$ is nef, or, all $F_j$ are nef and $L$ is ample,
			then 
			$$H^i(Y,R^qf_\ast(\omega_X(E))\otimes S^{k_1}F_1\otimes \cdots\otimes S^{k_l}F_l\otimes {\rm det}F_1\otimes \cdots\otimes {\rm det} F_l\otimes L)=0$$
			for any $i\geq1$ and $k_1 \geq 0,\dots, k_l\geq0$.
				\item If $F$ is Griffiths positive and ${\rm rank}(F)\geq 2$, then $$H^i(Y,R^qf_\ast (\omega_X(E))\otimes F^\ast\otimes ({\rm det}F)^k)=0$$ for any $i\geq 1$ and $k\geq \min\{m-i+1,{\rm rank}(F)\}$.
				\item If $0\rightarrow S\rightarrow F\rightarrow Q\rightarrow 0$ is an exact sequence of Hermitian vector bundles and $F>_k 0$, then 
				$$H^i(Y,R^qf_\ast (\omega_X(E))\otimes S\otimes ({\rm det}Q)^k)=0$$
				for any $i\geq 1$ and $k\geq \min\{m-i+1,{\rm rank}(S)\}$.
		\end{enumerate}
	\end{cor}
	This generalizes Kodaira-Nakano vanishing theorem \cites{Kodaira1953,Nakano1955}, Koll\'ar's vanishing theorem \cite{Kollar1986_1}, Ohsawa's vanishing theorem \cite{Ohsawa1984}, Griffiths's vanishing theorem \cite{Griffiths1969}, Liu-Sun-Yang vanishing theorems \cite{LSY2013},  some cases of Le Potier's vanishing theorem \cite{LP1975}, Demailly's vanishing theorem \cite{Demailly1988} and Manivel's vanishing theorem \cite{Manivel1997}. Further related works include \cites{LN2004,LN2005,Iwai2021,Mat2022,EV1987,Hoffman1989,Mat2016,Fujino2018,Inayama2020,LY2015}.
	
	There are two main inputs in the proof of Theorem \ref{thm_main}. The first involves the significant findings of Berndtsson \cite{Berdtsson2009} and Mourougane-Takayama \cite{MT2007} regarding the Nakano semipositivity of the higher direct image $R^qf_\ast(\omega_{X/Y}(E))$ over the dense Zariski open subset $Y^o$ of $Y$, where $f$ is a submersion over $Y^o$. The positivity of  higher direct image sheaves is of great importance in recent developments in complex algebraic geometry. Interested readers may refer to  \cites{Berdtsson2009,BP2008,BPW2022,Schumacher2012,Naumann2021,MT2008,Horing2010,Takayama2023,Viehweg2001} and the references therein. One of the main challenges in proving Theorem \ref{thm_main} is the presence of singular fibers. As a result, canonical metrics, such as the Hodge metric defined by Mourougane-Takayama \cite{MT2007}, on the torsion-free sheaf $R^qf_\ast(\omega_{X/Y}(E))$ have singularities along $Y\backslash Y^o$. This difficulty is overcome by establishing the $L^2$-Dolbeault resolution of $R^qf_\ast(\omega_{X}(E))$, which is the second input of the present article. The resolution is achieved by using differential forms on $Y^o$ that have locally finite $L^2$-norms at the boundary $Y\backslash Y^o$.  This resolution enables us to investigate  $R^qf_\ast(\omega_{X}(E))$ by analyzing the $L^2$-forms on the non-degenerate loci $Y^o$ of $f$. This technique builds upon the ideas developed in \cites{SZ2022,SZ2023}, which trace their roots to the proof of MacPherson's conjecture on the $L^2$-Dolbeault resolution of the Grauert-Riemenschneider sheaf \cite{Pardon_Stern1991,Ruppenthal2014} and the $L^2$-Dolbeault lemma established in the context of a variation of Hodge structure by Zucker \cite{Zucker1979}.
	
	\iffalse
	{\color{red}The MacPherson conjecture \cite{MacPherson1984}, which {\color{blue}stated that the and} was proved by Pardon-Stern \cite{Pardon_Stern1991} and generalized by Ruppenthal \cite{Ruppenthal2014}, builds a bridge between cohomology of the resolution of $Y$ and $L^2$ cohomology of $Y_{\rm reg}$.} An appropriate $L^2$ Dolbeault lemma, combined with the $L^2$ vanishing theorems on complete K\"ahler manifolds by Ohsawa \cite{Ohsawa1984}, leads us to the proof of the Koll\'ar type vanishing theorems for $R^qf_\ast\omega_{X}(E)$ when $E$ is a Nakano positive.
	\fi
	Let us explain the technique of the paper in more detail.
	Let $Y^o$ be the dense Zariski open subset of $Y$ such that $f^o:X^o\rightarrow Y^o$ is a proper holomorphic submersion, where $X^o$ denotes $f^{-1}(Y^o)$ and $f^o$ denotes $f|_{X^o}$. Then $R^qf^o_\ast(\omega_{X^o/{Y^o}}(E))\simeq R^qf_\ast(\omega_{X/Y}(E))|_{Y^o}$ is locally free \cite[Lemma 4.9]{MT2008} and admits a smooth Hodge  metric $h$ in the sense of Mourougane-Takayama \cite{MT2007} whose curvature is Nakano semipositive. 
	Let $ds_Y^2$ be a Hermitian metric on $Y$. 
	Let  $\sD^{m,k}_{Y}(R^qf^o_\ast(\omega_{X^o/{Y^o}}(E)))$ denote the sheaf of measurable $R^qf^o_\ast(\omega_{X^o/{Y^o}}(E))$-valued $(m,k)$-forms $\alpha$ such that $\alpha$ and its distributive $\dbar\alpha$ are locally square integrable near every point of $Y$ with respect to $ds_Y^2$ and the Hodge metric $h$ \cite{MT2007} on $R^qf^o_\ast(\omega_{X^o/{Y^o}}(E))$.  Define
	$$\sD^{m,\bullet}_{Y}(R^qf^o_\ast(\omega_{X^o/{Y^o}}(E)))=\sD^{m,0}_{Y}(R^qf^o_\ast(\omega_{X^o/{Y^o}}(E)))\stackrel{\dbar}{\to}\cdots\stackrel{\dbar}{\to}\sD^{m,m}_{Y}(R^qf^o_\ast(\omega_{X^o/{Y^o}}(E))),$$  the associated $L^2$-Dolbeault complex. 
	The main technical result of the present paper is the following  $L^2$-Dolbeault lemma.
	\begin{thm}\label{thm_main_resolution}
		$\sD_{Y}^{m,\bullet}(R^qf^o_\ast( \omega_{X^o/{Y^o}}(E)))$
		is a fine resolution of $R^qf_\ast(\omega_X(E))$ for every $q$.
	\end{thm}
    Theorem \ref{thm_main_resolution} holds for an arbitrary compact complex space $Y$. Readers may see \S 3 (especially Theorem \ref{thm_main_general}) for details.
    
	{\bf Notations:}
	\begin{enumerate}
		\item Let $X$ be a complex space. A \emph{Zariski closed} subset (=closed analytic subset) $Z$ of $X$ is a closed subset, that is locally defined as the zeros of a set of holomorphic functions. A subset $Y$ of $X$ is called \emph{Zariski open} if $X\backslash Y$ is Zariski closed.
		\item Two metrics $g_1$ and $g_2$ are said to be \emph{quasi-isometric} (written $g_1\sim g_2$) if there exists a constant $C$ such that $C^{-1}g_2\leq  g_1 \leq C g_2$.
		%	Let $\alpha$ and $\beta$ be functions, metrics or $(1,1)$-forms. We denote $\alpha\lesssim\beta$ if $\alpha\leq C\beta$ for some $C\in\bR_{>0}$. We say that $\alpha$ and $\beta$ are \emph{quasi-isometric} if $\alpha\lesssim\beta$ and $\beta\lesssim\alpha$,  denoted by $\alpha\sim \beta$.
	\end{enumerate}
	\section{Preliminary}
	\subsection{Hermitian vector bundle}
	Let $(M,ds_M^2)$ be a complex manifold of dimension $n$ with a Hermitian metric $ds_M^2$. 
	Let $(F,h_F)$ be a holomorphic vector bundle of rank $r$ on $M$ endowed with a Hermitian 
	metric $h_F$ and let $(F^\ast, h_F^\ast)$ be its dual Hermitian bundle. 
	Let $A^{p,q}(M,F)$ be the space of $F$-valued smooth $(p, q)$-forms on $M$ 
	and let $A_0^{p,q}(M,F)$ be its subspace  with compact support. 
	Let $\ast:A^{p,q}(M,F)\rightarrow A^{n-q,n-p}(M,F)$ be the Hodge star operator relative to $ds_M^2$ and let $\sharp_F:A^{p,q}(M,F)\to A^{q,p}(M,F^\ast)$ be the anti-isomorphism induced by $h_F$. Denote by $\langle-,-\rangle$ the pointwise inner product on $A^{p,q}(M,F)$. These operators are related by
	\begin{align}
		\langle\alpha,\beta\rangle{\rm vol}_{ds_M^2}=\alpha\wedge\ast\sharp_F\beta.
	\end{align}
	Let 
	\begin{align}
		(\alpha,\beta):=\int_{M}\langle\alpha,\beta\rangle{\rm vol}_{ds_M^2}
	\end{align}
	and $\|\alpha\|:=\sqrt{(\alpha,\alpha)}$.
	Let $\nabla=D'+\dbar$ be the Chern connection relative to $h_F$. Let
	$\dbar^\ast_{h_F}=-\ast D'\ast$ and $D'^\ast_{h_F}=-\ast\dbar\ast$ be the formal adjoints of $\dbar$ and $D'$ respectively.
	
	Denote by  $\Theta_{h_F}=\nabla^2$ the curvature of $(F,h_F)$. 
Locally we write 
$$\Theta_{h_F}=\sqrt{-1}\sum_{i,j}\omega_{ij}e_i\otimes e_j^\ast$$
where $\omega_{ij}\in A^{1,1}_{M}$, $(e_1,\dots,e_r)$ is an orthogonal local frame of $F$ and $(e^\ast_1,\dots,e^\ast_r)$ is the dual frame.

\begin{defn}\cite{Demailly2012}
\begin{itemize}
	\item A tensor $u\in T_M\otimes F$ is said to be \emph{of rank $m$} if $m$ is the smallest $\geq 0$ integer such that $u$ can be written as 
	$$u=\sum_{j=1}^m \xi_j\otimes s_j,\xi_j\in T_M,s_j\in F.$$
	\item $F$ is called \emph{$m$-positive} if $\sqrt{-1}\Theta_{h_F}(F)(u,u)>0$ for any nonzero $u\in T_M\otimes F$ of rank $\leq m$. In this case, we write $\Theta_{h_F}(F)>_m0$ (or $F>_m 0$).
\item 
$F$ is called \emph{Griffiths positive} if $m=1$, and  \emph{Nakano positive} if $m\geq \min\{n,r\}$.
\item $F$ is called \emph{Nakano semipositive}, if  the bilinear form
\begin{align*}
	\theta(u_1,u_2):=\sum_{i,j}\omega_{i,j}(u_{1i},\overline{u_{2j}}),\quad u_l=\sum_{i}u_{li}\otimes e_i\in T_{M}\otimes F,\quad l=1,2
\end{align*}
is semi-positive definite. 
\end{itemize}
\end{defn}

	\subsection{$L^2$-Dolbeault cohomology and $L^2$-Dolbeault complex}
	%	Let $(M,ds^2_M)$ be a Hermitian manifold of dimension $m$ and $(F,h_F)$ a Hermitian vector bundle on $M$. 
	Let $L^{p,q}_{(2)}(M,F)$ be the space of measurable $F$-valued $(p,q)$-forms on $M$ which are square integrable with respect to $ds^2_M$ and $h_F$. Although $L^{p,q}_{(2)}(M,F)$ depends on the choice of $ds^2_M$ and $h_F$, we will omit them in the notation when there is no confusion. Let $\dbar_{\rm max}$ denote the maximal extension of the $\dbar$ operator defined on the domains
	$$D^{p,q}_{\rm max}(M,F):=\textrm{Dom}^{p,q}(\dbar_{\rm max})=\{\phi\in L_{(2)}^{p,q}(M,F)|\dbar\phi\in L_{(2)}^{p,q+1}(M,E)\},$$
	where $\dbar$ is defined in the sense of distribution.
	The $L^2$ cohomology $H_{(2)}^{p,q}(M,F)$ is defined as the $q$-th cohomology of the complex $$D^{p,\bullet}_{\rm max}(M,F):=D^{p,0}_{\rm max}(M,F)\stackrel{\dbar_{\rm max}}{\to}\cdots\stackrel{\dbar_{\rm max}}{\to}D^{p,n}_{\rm max}(M,F).$$
	
	Let $Y$ be an irreducible complex space of dimension $m$, and let $Y^o$ be a dense Zariski open subset of its regular locus $Y_{\rm reg}$. Let $ds_Y^2$ be a Hermitian metric on $Y^o$ and let $(E,h)$ be a Hermitian vector bundle on $Y^o$. 
	Given an open subset $U$ of $Y$, the space $L_{Y}^{p,q}(E)(U)$ is defined as the space of measurable $E$-valued $(p,q)$-forms $\alpha$ on $U\cap Y^o$ such that for every point $x\in U$, there exists a neighborhood $V_x$ of $x$ in $Y$ such that 
	$$\int_{V_x\cap Y^o}|\alpha|^2_{ds_Y^2,h}{\rm vol}_{ds_Y^2}<\infty.$$
	For each $0\leq p,q\leq m$, we define the $L^2$-Dolbeault sheaf $\sD_{Y}^{p,q}(E)$ on $Y$ as follows:
	$$\sD_{Y}^{p,q}(E)(U):=\{\phi\in L_{Y}^{p,q}(E)(U)|\bar{\partial}_{\rm max}\phi\in L_{Y}^{p,q+1}(E)(U)\},\quad\forall \textrm{ open subset }U\textrm{ of }Y.$$
	Now the $L^2$-Dolbeault complex of sheaves $\sD_{Y}^{p,\bullet}(E)$ is defined as:
	\begin{align*}
		\sD_{Y}^{p,0}(E)\stackrel{\dbar}{\to}\sD_{Y}^{p,1}(E)\stackrel{\dbar}{\to}\cdots\stackrel{\dbar}{\to}\sD_{Y}^{p,m}(E),
	\end{align*}
	where $\dbar$ is taken in the sense of distribution. 
	%We omit the metrics in the notation when there is no confusion.
		\begin{rmk}
			The $L^2$ cohomology and the $L^2$-Dolbeault sheaf are invariants of the quasi-isometry class of $ds_M^2$, $h_F$, $ds_Y^2$ and $h$.
	\end{rmk}
	\begin{defn}\label{defn_Hermitian_metric}
		A Hermitian metric $ds_0^2$ on $Y^o$ is called a \emph{Hermitian metric on $Y$} if, for every $x\in Y$,  there exists  a neighborhood $U$ of $x$ in $Y$ and a holomorphic closed immersion $U\subset V$ into a complex manifold such that $ds_0^2|_{U\cap Y^o}\sim ds^2_V|_{U\cap Y^o}$ for some Hermitian metric $ds^2_V$ on $V$. If the $(1,1)$-form associated with $ds_0^2$ is moreover $d$-closed on $Y^o$, we then call $ds_0^2$ a \emph{K\"ahler  metric on $Y$}.
	\end{defn}
	The $L^2$-Dolbeault sheaf with respect to a  Hermitian metric $ds_0^2$ on $Y$ is always fine, as shown by the following lemma.
	\begin{lem}\cite[Lemma 2.2]{SZ2022}\label{lem_fine_sheaf}
		Suppose that for every point $x\in {Y\setminus Y^o}$ there exists a neighborhood $U_x$ of $x$ in $Y$ and a Hermitian metric $ds^2_0$ on $U_x$ such that $ds^2_0|_{Y^o\cap U_x}\lesssim ds_Y^2|_{Y^o\cap U_x}$. Then the $L^2$-Dolbeault sheaf $\sD^{p,q}_{Y}(E)$ with respect to $ds_Y^2$ and $h_E$ is a fine sheaf for every $p$ and $q$.
	\end{lem}
	
	\subsection{Harmonic theory on higher direct images}
	In this section, we briefly review the harmonic theory on higher direct images presented in \cite{SZ2022}. This theory is a generalization of Takegoshi's work \cite{Takegoshi1995} to complex spaces and will be used in proving Theorem \ref{thm_main_resolution}.  
	
	Let $f:X\rightarrow Y$ be a proper surjetive holomorphic morphism from a compact K\"ahler manifold  to an irreducible analytic space with $\dim X=n+m$ and $\dim Y=m$ respectively. Let $Y^o$ be the dense Zariski open subset of the loci $Y_{\rm reg}$ of the regular points of $Y$ such that $f^o:=f|_{X^o}:X^o:=f^{-1}(Y^o)\rightarrow Y^o$ is a proper holomorphic submersion. Let $(E,h_E)$ be a Nakano semipositive holomorphic vector bundle on $X$. As \cite[Lemma 2.14]{SZ2022}, we fix a K\"ahler metric $ds^2$ on $X^o$ 
	such that the following conditions hold.
	\begin{enumerate}
		\item For every point $x\in X$ there is a neighborhood $U$ of $x$, a function $\Phi\in C^{\infty}(U\cap X^o)$ such that $|\Phi|+|d\Phi|_{ds^2}<\infty$ and $ds^2|_{U\cap X^o}\sim\sqrt{-1}\ddbar\Phi$.
		\item $ds^2$ is locally complete on $X$, i.e.,  there exists for every point $x\in X$ a neighborhood $U$ of $x$ such that $(\overline{U}\cap X^o,ds^2)$ is complete.
		\item $ds^2$ is locally bounded from below by a Hermitian metric, i.e., there exists, for every point $x\in X$, a neighborhood $U$ of $x$ and a Hermitian metric $ds^2_0$ on $U$ such that $ds^2_0|_U\lesssim ds^2|_U$.
	\end{enumerate}
	%The proof of the existence of such metric when $X$ is compact can be found in \cite[Lemma 2.14]{SZ2022}. 
	Let $\omega$ denote the K\"ahler form of $ds^2$.	
	Let $U\subset Y$ be a Stein open subset. Let $\sP(f^{-1}(U))$ be the nonempty set of $C^\infty$ plurisubharmonic functions $\varphi:f^{-1}(U)\to(-\infty,c_\ast)$ for some $c_\ast\in(-\infty,\infty]$ such that 
	$\{z\in f^{-1}(U)|\varphi(z)<c\}$ is precompact in $f^{-1}(U)$ for every $c<c_\ast$. 
	%Notice that $\sP(f^{-1}(U))$ is non-empty since $U$ is Stein.
	For every $C^\infty$ plurisubharmonic function $\varphi\in \sP(f^{-1}(U))$, set the subspace of $E$-valued $L^2$ harmonic $(n+m,q)$-forms with respect to $\omega$ and $h_E$: 
	\begin{align}\label{align_harmonic}
		\sH^{m+n,q}(f^{-1}(U),E,\varphi):=\left\{\alpha\in\sD^{m+n,q}_{X}(E)(f^{-1}(U))\big|\dbar\alpha=\dbar^\ast_{h_E}\alpha=0, e(\dbar\varphi)^\ast\alpha=0\right\},
	\end{align}
	where $e(\dbar \varphi)^\ast$ denotes the adjoint operator of the left exterior product acting on $A^{m+n,q}(f^{-1}(U),E)$ by a form $\dbar \varphi\in A^{0,1}(f^{-1}(U))$ with respect to the inner product induced by $h_E$. We would like to point out that the equalities on the right-hand side of (\ref{align_harmonic}) are only required to hold on $X^o\cap f^{-1}(U)$, not on the whole of $f^{-1}(U)$.
	By the regularity theorem for elliptic operators of second order, every element of $\sH^{m+n,q}(f^{-1}(U),E,\varphi)$ is $C^\infty$ on $X^o\cap f^{-1}(U)$.
	
	Let $\varphi,\psi\in \sP(f^{-1}(U))$  be arbitrary $C^\infty$ plurisubharmonic functions. Then $$ \sH^{m+n,q}(f^{-1}(U),E,\varphi)=\sH^{m+n,q}(f^{-1}(U),E,\psi)$$ for every $q\geq 0$ (\cite[Proposition 3.4(2)]{SZ2022}). Therefore we 
	%omit the $C^\infty$ plurisubharmonic function $\psi$ and 
	use the notation $\sH^{m+n,q}(f^{-1}(U),E)$ instead of $\sH^{m+n,q}(f^{-1}(U),E,\psi)$ in the sequel.
	
	According to \cite[Proposition 3.7]{SZ2022}, the restriction map 
	$$\sH^{m+n,q}(f^{-1}(V),E)\to \sH^{m+n,q}(f^{-1}(U),E)$$
	is well defined for any pair of Stein open subsets $U\subset V\subset Y$. Hence the data
	\begin{align*}
		U\mapsto\sH^{m+n,q}(f^{-1}(U),E),\quad \forall \textrm{  Stein open subset } U\textrm{ of } Y
	\end{align*}
	forms a presheaf on $Y$. We denote by $\sH^{m+n,q}_f(E)$ its
	sheafification. 
	\begin{thm}\cite[Theorem 3.8]{SZ2022}\label{thm_harmonic}
		$\sH^{m+n,q}_f(E)$ is a sheaf of $\sO_Y$-modules, and there exists a natural isomorphism
		$$\tau_f:R^qf_\ast(\omega_X(E))\rightarrow \sH^{m+n,q}_f(E)$$
		of $\sO_Y$-modules for every $q\geq0$. Moreover 
		$$\sH^{m+n,q}_f(E)(U)=\sH^{m+n,q}(f^{-1}(U),E)$$
		for any Stein open subset $U\subset Y$.
	\end{thm}
	For every $0\leq p\leq m+n$, we define the $L^2$-Dolbeault sheaf $\sD^{p,0}_{X}(E)$ with respect to $\omega$ and $h_E$ as in \S2.2, and define a subsheaf $\Omega^{p}_{X,(2)}(E)$ as 
	$$\Omega^{p}_{X,(2)}(E)(U)=\left\{\alpha\in\sD^{p,0}_{X}(E)(U)\bigg|\dbar\alpha=0\right\},\quad \forall \textrm{ open subset } U\textrm{ of } X.$$
	
	\begin{prop}\label{prop_Hodge_star}
		The Hodge star operator $\ast$ relative to $\omega$ yields a splitting homomorphism 
		$$\delta^q : R^q f_\ast(\omega_X(E))\overset{\tau_f}\simeq \sH^{m+n,q}_f(E)\xrightarrow{\ast} f_\ast (\Omega_{X,(2)}^{m+n-q}(E))$$
		with $\sL^q\circ \delta^q={\rm Id}$ for the homomorphism
		$$\sL^q:f_\ast(\Omega_{X,(2)}^{m+n-q}(E))\rightarrow \sH^{m+n,q}_f(E)\simeq R^qf_\ast(\omega_X(E))$$
		induced by the $q$-times left exterior product by $\omega$. Moreover, the image of $\delta^q|_{Y^o}$ lies in $f^o_\ast(\Omega^{n-q}_{X^o}(E)\otimes f^\ast\Omega^m_{Y^o})$.
	\end{prop}
	\begin{proof}
		See \cite[Proposition 3.7]{SZ2022} and the proof of Theorem 4.1 in \cite{SZ2022}.
	\end{proof}

	\section{$L^2$-Dolbeault resolution of the higher direct image sheaf}
	Let $f:X\rightarrow Y$ be a proper holomorphic surjective morphism from a compact K\"ahler manifold to an irreducible complex space, where $\dim X=n+m$ and $\dim Y= m$. Let $(E,h_E)$ be a Nakano semipositive vector bundle on $X$, and let $ds_Y^2$ be a Hermitian metric on $Y$. 
	Let $Y^o$ be the dense Zariski open subset of $Y_{\rm reg}$ such that $f^o:X^o\rightarrow Y^o$ is a proper holomorphic submersion, where $X^o$ denotes $f^{-1}(Y^o)$ and $f^o$ denotes $f|_{X^o}$. 	
	According to \cite[Lemma 4.9]{MT2008}, $R^qf^o_\ast(\omega_{X^o/{Y^o}}(E))\simeq R^qf_\ast(\omega_{X/Y}(E))|_{Y^o}$ is locally free. Here, $ds^2$ is a K\"ahler metric on $X^o$ as described in \S2.3, with $\omega$ being its associated K\"ahler form.
	
	\subsection{Mourougane-Takayama's Hodge metric on  $R^qf^o_\ast(\omega_{X^o/{Y^o}}(E))$}
	This section provides a review of  Mourougane-Takayama's construction of a Hodge metric on $R^qf^o_\ast(\omega_{{X^o}/{Y^o}}(E))$ with Nakano semipositive curvature \cite{MT2008}. For more details, see \cites{MT2007,MT2008,MT2009}.
	
	%Let $\omega$ be the relative K\"ahler form for $f$.
	
	% and a $C^\infty$ psh exhaustion function $\varphi=f^\ast \sum_{j=1}^m|t_j|^2$ on $X^o_W$, where $X^o_W$ denotes $(f^o)^{-1}(W)$.
	%Take any $0\leq q\leq n$. 
	%Following \cite{Takegoshi1995}[4.3],
	%set the subspace of $E$-valued harmonic $(n+m,q)$-forms with respect to $\omega$ and $h$:
	%	$$\sH^{n+m,q}(X_W,E,\varphi)=\{u\in A^{n+m,q}(X_W,\sE)\mid \dbar u=\dbar_{h}^\ast u=0\quad {\textrm and}\quad e(\dbar \varphi)^\ast u=0\quad{\textrm{on}}\quad X_W\}.$$
	Define the sheaf $\sH_{f}^{n+m,q}(E,h)$ associated with the proper submersion $f:X\to Y$ as described in \S2.3. It follows from Theorem \ref{thm_harmonic} that there exists a natural isomorphism $$\tau: R^qf_\ast(\omega_{X}(E))\simeq\sH_{f}^{n+m,q}(E).$$
	Denote $$\tau^o:=\tau|_{Y^o}:R^qf^o_\ast(\omega_{X^o}(E))\simeq \sH_{f}^{n+m,q}(E)|_{Y^o}=:\sH_{f^o}^{n+m,q}(E).$$
	Let $y\in Y^o$ and let $W\simeq\{t=(t_1,\dots,t_m)\in\bC^m\mid\|t\|<1\}$ be holomorphic coordinates centered at $y$.
	Let $X^o_W$ denote $(f^o)^{-1}(W)$ and let $dt$ denote $dt_1\wedge\cdots\wedge dt_m$.
	Take a trivialization $\sO_{W}\xrightarrow{\sim}\Omega_{W}^m$ given by $1\mapsto dt$.
	This trivialization induces an isomorphism of sheaves $\Omega_{{X^o_W}/{W}}^n\simeq \Omega_{{X^o_W}/{W}}^n\otimes (f^o)^\ast \Omega_W^m\simeq \omega_{X^o_W}$ via $u\mapsto u\wedge dt$. Consequently, the isomorphism extends to higher direct image sheaves as follows: $$\alpha^q_W: R^qf^o_\ast (\omega_{{X^o}/{Y^o}}(E))|_W\xrightarrow{\simeq}R^qf^o_\ast(\omega_{X^o}(E))|_W.$$
	We also have an injection $\Omega_{{X^o_W}/W}^{n-q}\rightarrow \Omega_{X^o_W}^{n+m-q}$ by $\sigma\mapsto \sigma\wedge dt$.
	This injection induces the injection $$\beta_W:f^o_\ast(\Omega_{{X^o}/{Y^o}}^{n-q}(E))|_W\rightarrow f^o_\ast(\Omega_{X^o}^{n+m-q}(E))|_W.$$ Notice that for every $u\in \sH^{n+m,q}(X^o_W,E)$, there exists $\sigma_u\in H^0(X^o_W,\Omega_{{X^o_W}/W}^{n-q}(E))$ such that $\ast u=\sigma_u\wedge dt$ (see Proposition \ref{prop_Hodge_star} or \cite[Theorem 4.2-(3)]{MT2008}). Therefore, the map $u \mapsto \sigma_u$ is well-defined and injective, and thus yields a homomorphism $$\delta^q_W:\sH_{f^o}^{n+m,q}(E)|_W\rightarrow f^o_\ast(\Omega_{{X^o}/{Y^o}}^{n-q}(E))|_W,$$ where $$\ast=\beta_W\circ \delta^q_W:\sH_{f^o}^{n+m,q}(E)|_W\rightarrow f^o_\ast(\Omega_{X^o/{Y^o}}^{n-q}(E))|_W\rightarrow f^o_\ast (\Omega_{X^o}^{n+m-q}(E))|_W.$$ 
	Then the composition map
	$$S_{W}^q:R^qf^o_\ast( \omega_{X^o/{Y^o}}(E))|_W\xrightarrow{\alpha^q_W}R^qf^o_\ast(\omega_{X^o}(E))|_W\xrightarrow{\tau^o} \sH_{f^o}^{n+m,q}(E)|_W\xrightarrow{\delta_W^q} f^o_\ast (\Omega_{{X^o}/{Y^o}}^{n-q}(E))|_W$$
	is injective.
	
	For every $y\in W$ and every pair of vectors $u_y, v_y\in R^qf^o_\ast(\omega_{X^o/{Y^o}}(E))_y$, Mourougane-Takayama defined 
	$$h(u_y, v_y)= \frac{c_{n-q}}{q!}\int_{f^{-1}(y)}(\omega^q\wedge S^q_{W}(u_y) \wedge_{h_E}\overline{S^q_{W}(v_y))}|_{f^{-1}(y)},\quad c_{n-q}=\sqrt{-1}^{(n-q)^2}
	.$$
	The induced metric $h$ is independent of the choice of the coordinate $W$ and thus defines a global Hermitian metric on the bundle $R^qf^o_\ast(\omega_{X^o/{Y^o}}(E))$ (\cite[Lemma 5.2]{MT2008}). This Hermitian metric $h$ is then called the \emph{Hodge metric} on $R^qf^o_\ast(\omega_{X^o/{Y^o}}(E))$.
	\begin{thm}\emph{(\cite[Theorem 1.1]{MT2008})}\label{thm_Nsemip}
		$\sqrt{-1}\Theta_{h}(R^qf^o_\ast (\omega_{X^o/{Y^o}}(E)))$ is Nakano semipositive.
	\end{thm}
	Now we define $\sD_{Y}^{m,\bullet}(R^qf^o_\ast(\omega_{X^o/{Y^o}}(E)))$ as the associated $L^2$-Dolbeault complex 
	%of $R^qf^o_\ast(\omega_{X^o/{Y^o}}(E))$-valued forms on $Y^o$ 
	with respect to $ds_Y^2$ and $h$. The main result concerning this complex is the following.
	\begin{thm}\label{thm_main_general}
		$\sD_{Y}^{m,\bullet}(R^qf^o_\ast( \omega_{X^o/{Y^o}}(E)))$
		is a fine resolution of $R^qf_\ast(\omega_X(E))$ for every $q$.
	\end{thm}
	\subsection{Exactness of $\sD_{Y}^{m,\bullet}(R^qf^o_\ast(\omega_{X^o/{Y^o}}(E)))$}
	Now let us introduce the following $L^2$ estimate, which is essentially due to H\"ormander \cite{Hormander1990} and Andreotti-Vesentini \cite{AV1965}.  Here we use the version suitable for our purpose as stated in \cites{Demailly1982,Demailly2012,Ohsawa1980,Ohsawa1984}.
	\begin{thm}\emph{(\cite{Demailly2012}, \cite[Theorem 5.1]{Demailly1982} and \cite{Ohsawa1980,Ohsawa1984})}\label{thm_Hormander_incomplete}
		Let $M$ be a complex manifold of dimension $n$ that admits a complete K\"ahler metric. Let $(F,h_F)$ be a Hermitian holomorphic vector bundle on $M$ such that $$\sqrt{-1}\Theta_{h_F}(F)\geq \omega_0\otimes {\rm Id}_{F}$$ for some (not necessarily complete) K\"ahler form $\omega_0$ on $M$. Then for every $q>0$ and every $\alpha\in L^{n,q}_{(2)}(M,F;\omega_0,h_F)$ such that $\dbar\alpha=0$, there exists $\beta\in L^{n,q-1}_{(2)}(M,F;\omega_0,h_F)$ such that $\dbar\beta=\alpha$ and $\|\beta\|^2_{\omega_0,h_F}\leq q^{-1}\|\alpha\|^2_{\omega_0,h_F}$.
	\end{thm}
	The above theorem works effectively locally on complex analytic singularities due to
	the following lemma by Grauert \cite{Grauert1956} (see also \cite[Lemma 2.4]{Pardon_Stern1991}).
	\begin{lem}\label{lem_complete_metric_exists_locally}
		Let $x$ be a point in a complex analytic space $Y$ and let $Y^o$ be  a dense
		Zariski open subset of $Y_{\rm reg}$. Then there exists a neighborhood $U$ of $x$ in $Y$ and a complete K\"ahler metric on $U \cap Y^o$. 
	\end{lem}
	%	Recall that $ds^2_Y$ is a Hermitian metric on $Y$ and $h$ is the Hodge metric on $R^qf^o_\ast(\omega_{X^o/{Y^o}}(E))$.
	The main purpose of this subsection is the following theorem.
	\begin{thm}\label{thm_excatness}
		The complex of sheaves $\sD^{m,\bullet}_{Y}(R^qf^o_\ast(\omega_{X^o/{Y^o}}(E)))$
		is exact at $\sD^{m,q}_{Y}(R^qf^o_\ast(\omega_{X^o/{Y^o}}(E)))$ for every $q>0$. 
	\end{thm}
	\begin{proof}	
		Since the problem is local, we may assume that $Y$ is a germ of complex	analytic space and $ds_Y^2$ is quasi-isometric to some K\"ahler form $\sqrt{-1}\ddbar \Phi$, where $\Phi$ is some bounded $C^\infty$ strictly plurisubharmonic function on $Y$.  Thus $ C\sqrt{-1}\partial\dbar\Phi\geq\omega_Y$ for some constant $C>0$, where $\omega_Y$ denotes the K\"ahler form associated with $ds_Y^2$.
		Let $h'=e^{-C\Phi}h$ be a modified Hermitian metric on $R^qf^o_\ast(\omega_{X^o/{Y^o}}(E))$. Theorem \ref{thm_Nsemip} yields that 
		$$\sqrt{-1}\Theta_{h'}(R^qf^o_\ast(\omega_{X^o/{Y^o}}(E))) = C\sqrt{-1}\partial\dbar\Phi\otimes{\rm Id}+\sqrt{-1}\Theta_{h}(R^qf^o_\ast(\omega_{X^o/{Y^o}}(E)))\geq \omega_Y\otimes{\rm Id}.$$
		Since $Y$ is compact, we may assume that $Y^o$ admits a complete K\"ahler metric by using Lemma \ref{lem_complete_metric_exists_locally}. As $\Phi$ is bounded, we have $h' \sim h$. By Theorem \ref{thm_Hormander_incomplete}, we obtain that 
		$$H^{m,q}_{(2)}(Y^o, R^qf^o_\ast(\omega_{X^o/{Y^o}}(E));\omega_Y,h)=H^{m,q}_{(2)}(Y^o, R^qf^o_\ast(\omega_{X^o/{Y^o}}(E));\omega_Y,h')=0,\quad \forall q>0,$$
		which proves the theorem.
	\end{proof}
	\subsection{Proof of Theorem \ref{thm_main_resolution}}
	Recall that $ds_Y^2$ is a Hermitian metric on $Y$. It follows from Lemma \ref{lem_fine_sheaf} that all $\sD_{Y}^{m,i}(R^qf^o_\ast(\omega_{X^o/{Y^o}}(E)))$ are fine sheaves.
	To prove Theorem \ref{thm_main_resolution} it remains to show the following theorem.
	\begin{thm}\label{thm_kernel}
		There is an isomorphism between $R^qf_\ast(\omega_X(E))$ and 
		$${\rm ker}(\dbar:\sD_{Y}^{m,0}(R^qf^o_\ast(\omega_{X^o/{Y^o}}(E)))\to \sD_{Y}^{m,1}(R^qf^o_\ast(\omega_{X^o/{Y^o}}(E)))).$$
	\end{thm}
	%By using Theorem \ref{thm_kernel}, Theorem \ref{thm_excatness} and Lemma \ref{lem_fine_sheaf}, we are able to prove Theorem \ref{thm_main_resolution}. 
	%	 The remainder of this section will be dedicated to the proof of Theorem \ref{thm_kernel}. 
	The main idea is to regard both sheaves as subsheaves of $j_\ast(R^qf^o_\ast(\omega_{X^o}(E)))$ where $j:Y^o\to Y$ is the immersion, and show that the sections of both sheaves share the same boundary conditions. For the sake of convenience, we let $\sK={\rm ker}(\dbar:\sD_{Y}^{m,0}(R^qf^o_\ast(\omega_{X^o/{Y^o}}(E)))\to \sD_{Y}^{m,1}(R^qf^o_\ast(\omega_{X^o/{Y^o}}(E))))$.
	\subsubsection{Boundary condition of $R^qf_\ast(\omega_X(E))$}
	Since $R^qf_\ast(\omega_X(E))$ is torsion free (\cite{Takegoshi1995}), there exists a natural embedding $R^qf_\ast(\omega_X(E))\subset j_\ast(R^qf^o_\ast(\omega_{X^o}(E)))$.
	Let $\sH_f^{m+n,q}(E)$ be the sheaf defined in \S2.3 with respect to $ds^2$ and $h_E$.
	Then Theorem \ref{thm_harmonic} yields a natural isomorphism $$\tau_f: R^qf_\ast(\omega_X(E))\simeq \sH_f^{m+n,q}(E).$$
	Therefore for any Stein open subset $U$ of $ Y$ and any section $s\in R^qf^o_\ast(\omega_{X^o}(E))(U\cap Y^o)$,  $s$ can be extended to a section of  $R^qf_\ast(\omega_X(E))(U)$ if and only if $s$ satisfies the following boundary condition:
	
	{\bf (I):} $\tau_f(s)$ is an $E$-valued harmonic $(m+n,q)$-form on $f^{-1}(U)\cap X^o$ which is locally $L^2$ at every point of $f^{-1}(U)$ (with respect to $ds^2$ and $h_E$) and satisfies that  $e(\dbar\varphi)^\ast\tau_f(s)=0$ on $f^{-1}(U)\cap X^o$ for any $\varphi\in\sP(f^{-1}(U))$.
	\subsubsection{Boundary condition of $\sK$}
	By the classical Dolbeault Lemma one has
	$R^qf^o_\ast (\omega_{X^o}(E))=\sK|_{Y^o}$. As a result, there is a natural embedding $\sK\subset j_\ast(R^qf^o_\ast (\omega_{X^o}(E)))$. For any Stein open subset $U$ of $Y$ and any section $s\in R^qf^o_\ast(\omega_{X^o}(E))(U\cap Y^o)$,   $s$ can be extended to a section of $\sK(U)$ if and only if $s$ satisfies the following boundary condition:
	
	{\bf (II):} $s$ is locally $L^2$ (with respect to $ds_Y^2$ and the Hodge metric $h$) at every point of $U$. 
	\subsubsection{Comparison of the boundary conditions}
	Let $U$ be a Stein open subset of $Y$ and let $s$ be a section of $R^qf^o_\ast(\omega_{X^o}(E))(U\cap Y^o)$. We are going to show that $s$ satisfies Condition {\bf (I)} if and only if it satisfies Condition {\bf (II)}. First, we need the following lemma.
	\begin{lem}\label{lem_L2equality}
		For every open subset $U$ of $Y$ and every $s\in R^qf^o_\ast(\omega_{X^o}(E))(U\cap Y^o)$, it holds that 
		\begin{align}\label{align_L2equi}
			\int_{U\cap Y^o}|s|_{ds_Y^2,h}^2{\rm vol}_{ds_Y^2}=\frac{c_{n-q}}{c_{n+m-q}}\int_{f^{-1}(U)\cap X^o}|\tau_f(s)|_{ds^2,h_E}^2{\rm vol}_{ds^2},\quad c_d=\sqrt{-1}^{d^2}.
		\end{align}
	\end{lem}
	%{\color{red}\begin{lem}\label{lem_hodgestar}
			%		Let $u\in A^{n,q}(X,E)$. Then $u=\frac{c_{n-q}}{q!}\omega^q\wedge \ast u$. 
			%\end{lem}}
			
			\begin{proof}
				By using a partition of unity on $U\cap Y^o$ we may assume that $W=U\cap Y^o$ is small enough that satisfies the following conditions.
				\begin{itemize}
					\item There is a  holomorphic global coordinate $t=(t_1,\dots,t_m)$ on $W$ such that $(W;t)$ is a unit ball in $\bC^m$.
					\item There is a finite set of holomorphic local coordinates $ \{(U^\alpha; z^\alpha, t)\}_{\alpha\in I}$ of $f^{-1}(W)$ such that $f^{-1}(W)\subset \cup_{\alpha\in I}U^\alpha$ and $f|_{U^\alpha}$ is defined by $(z^\alpha,t)\mapsto t$. Namely $z^\alpha=(z^\alpha_1,\dots,z^\alpha_n)$ are holomorphic local coordinates on the fiber $f^{-1}(\{t=0\})$.
					\item There is a partition of unity $1=\sum_{\alpha\in I}\rho_\alpha$ on $f^{-1}(W)$ such that ${\rm supp}(\rho_\alpha)\subset U^\alpha$ for every $\alpha\in I$.
				\end{itemize}
				Let $s=dt\otimes u$ where $dt:=dt_1\wedge \cdots\wedge dt_m$ is a local frame of $\omega_{W}$ and $u$ is a section of $R^qf^o_\ast(\omega_{X^o/Y^o}(E))(W)$.
				Some computations yield that 
				\begin{align}
					\int_{W}|s|_{ds_Y^2,h}^2{\rm vol}_{ds_Y^2}&
					=\int_{W} |u|_{h}^2dt\wedge d\bar{t}\\\nonumber
					&=\int_{y\in W}dt\wedge d\bar{t}\int_{f^{-1}\{y\}}\frac{c_{n-q}}{q!}(\omega^q\wedge S^q_W(u_y)\wedge_h\overline{S^q_W(u_y)})|_{f^{-1}\{y\}}\\\nonumber
					&=\sum_{\alpha\in I}\frac{c_{n-q}}{q!}\int_{U^\alpha}\rho_\alpha\omega^q\wedge (\ast\tau_f(s))\wedge_h \overline{ (\ast\tau_f(s))}\quad\textrm{(Fubini theorem)}\\\nonumber
					&= \sum_{\alpha\in I}\frac{c_{n-q}}{c_{n+m-q}}\int_{U^\alpha}\rho_\alpha \tau_f(s)\wedge_h\overline{\ast \tau_f(s)}\\\nonumber
					&=\frac{c_{n-q}}{c_{n+m-q}}\int_{f^{-1}(W)\cap X^o}|\tau_f(s)|_{ds^2,h_E}^2{\rm vol}_{ds^2}.\\\nonumber
				\end{align}
			\end{proof}
			Notice that $U\cap Y^o$ may not be Stein. It follows from Theorem \ref{thm_harmonic} that $\tau_f(s)$ is an $E$-valued harmonic $(m+n,q)$-form on $f^{-1}(U)\cap X^o$ such that there exists a covering $U\cap Y^o=\cup V_i$ of Stein open subsets $V_i$ and $\varphi_i\in\sP(f^{-1}(V_i))$ such that $e(\dbar\varphi_i)^\ast \tau_f(s|_{V_i})=0$ for every $i$.
			
			Now we assume that $s$ satisfies Condition {\bf (I)}.
			It follows from Lemma \ref{lem_L2equality} that $s$ is locally $L^2$ at every point of $U$ (with respect to $ds^2$ and $h_E$) if and only if $\tau_f(s)$ is locally $L^2$ (with respect to $ds_Y^2$ and $h$) at every point of $f^{-1}(U)$. This shows that $s$ satisfies Condition {\bf (II)}.
			
			To prove the converse, we assume that $s$ satisfies Condition {\bf (II)}.
			Lemma \ref{lem_L2equality} shows that $\tau_f(s)$ is a harmonic form which is locally $L^2$ at every point of $f^{-1}(U)$. Notice that $\ast\tau_f(s)\in \Gamma(X^o\cap f^{-1}(U),\Omega^{n-q}_{X^o\cap f^{-1}(U)}(E)\otimes f^\ast\Omega^m_{U\cap Y^o})$ by  Proposition \ref{prop_Hodge_star}. Consequently, $f^\ast(\dbar \varphi)\wedge\ast\tau_f(s)=0,\forall \varphi\in \sP(f^{-1}(U))$ for the reason of bi-degree. Thus, $e(\dbar \varphi)^\ast\tau_f(s)=0$, which indicates that $s$ satisfies Condition {\bf (I)}.
			The proof of Theorem \ref{thm_kernel} is now complete.
			\section{Applications to Koll\'ar type vanishing theorems}
To establish the main vanishing theorem, it is necessary to introduce the following estimate.	\begin{lem}\cite{Demailly2012}\label{lem_demaillyestimate}
				Let $(M, \tilde{\omega})$ be a complete K\"ahler manifold of dimension $n$, $\omega$ another K\"ahler metric,
				possibly non-complete, and $E$ a $m$-semi-positive vector bundle of rank $r$ on $M$. Let $g\in L_{(2)}^{n,q}(X,E)$ be such that $D''g = 0$ and
				$\int_M\langle A_q^{-1}g,g\rangle dV<+\infty$
				with respect to $\omega$, where $A_q=[i\Theta(E),\Lambda]$ in bidegree $(n, q)$ and $q \geq1$,
				$m \geq \min\{n-q + 1, r\}$. Then there exists $f\in L_{(2)}^{n,q-1}(M,E)$  such that $D''f = g$ and
				$$\|f\|^2\leq \int_M\langle A_q^{-1}g,g\rangle dV.$$
			\end{lem}
			\begin{thm}\label{thm_main_proof}
			Let $f:X\rightarrow Y$ be a proper holomorphic surjective morphism from a compact K\"ahler manifold to a compact K\"ahler manifold. Let $E$ be a Nakano semipositive vector bundle on $X$ and let $(F,h_F)$ be a $k$-positive Hermitian vector bundle on $Y$ of rank $r$. Then
			$$H^i(Y,R^qf_\ast(\omega_X(E))\otimes F)=0,\quad\forall i\geq 1,k\geq\min\{\dim_{\bC}Y-i+1,r\}.$$
			\end{thm}
			\begin{proof}
				First, we assert that $$R^qf_\ast\omega_X(E)\otimes F\rightarrow \sD^{\dim_{\bC},\bullet}_{Y}(R^qf^o_\ast(\omega_{X^o/{Y^o}}(E\otimes f^\ast F)))$$ is a fine resolution for each $q$.
				Since the problem is local, we consider an arbitrary point $y\in Y$ and let $V$ be an open neighborhood around $y$ in $Y$ so that $F|_V\simeq \sO_V^{\oplus r}$ and the metric $h_F$ is quasi-isometric to the trivial metric.
				Consequently, $E \otimes f^\ast F$ is Nakano semipositive on $V$. By applying Theorem \ref{thm_main_resolution}, we confirm the validity of the claim.

		Thus there exists an isomorphism
			$$H^i(Y,R^qf_\ast(\omega_X(E))\otimes F)\simeq H^{\dim_{\bC}Y,i}_{(2)}(Y^o,R^qf_\ast(\omega_{X/Y}(E\otimes f^\ast(F)))|_{Y^o};ds_Y^2,h\otimes h_F),\quad \forall i,$$
			where $h$ is the Hodge metric on $R^qf^o_\ast(\omega_{X^o/Y^o}(E))$ and 
			$$R^qf_\ast(\omega_X(E\otimes f^\ast(F)))|_{Y^o}\simeq R^q f^o_{\ast}(\omega_{X^o/Y^o}(E))\otimes F|_{Y^o}.$$
			
				%is endowed with the metric $hh_F$, 
					Since $F$ is $k$-positive, the Hermitian operator $[\sqrt{-1}\Theta_{h_F}(F),\Lambda]$ is positive on 
				$\Lambda^{\dim_{\bC}Y,i}T^\ast_Y\otimes F$ for each $i\geq 1,k\geq\min\{\dim_{\bC}Y-i+1,r\}$
				( \cite[Chap VII, Lemma (7.2)]{Demailly2012}).
				Since $Y$ is compact, by applying Theorem \ref{thm_Nsemip} we can conclude that $[\sqrt{-1}\Theta_{h\otimes h_F}(R^q f^o_{\ast}(\omega_{X^o/Y^o}(E))\otimes F),\Lambda]$ has a positive lower bound on 
				$(\Lambda^{\dim_{\bC}Y,i}T^\ast_Y\otimes R^q f^o_{\ast}(\omega_{X^o/Y^o}(E))\otimes F)$ for each $i\geq 1,k\geq\min\{\dim_{\bC}Y-i+1,r\}$. %Moreover, since $Y$ is compact, there exists a K\"ahler form $\omega$ on $Y^o$ such that
				% $$\sqrt{-1}\Theta_{hh_F}(R^q f^o_{\ast}(\omega_{X^o/Y^o}(E))\otimes F|_{Y^o})\geq \omega.$$ 
				Additionally, the compactness of $Y$ implies that there is a globally defined complete K\"ahler metric on $Y^o$, as shown in \cite[Proposition 3.2]{Zucker1979}. By applying Lemma \ref{lem_demaillyestimate}, we obtain the result that $$H^{\dim_{\bC}Y,i}_{(2)}(Y^o,R^qf_\ast(\omega_{X/Y}(E\otimes f^\ast(F)))|_{Y^o})=0,\quad i>0.$$ Therefore, the theorem is proven. 
			\end{proof}
			Assume that $Y$ is a non-singular projective $m$-fold and let $F,F_1,\dots,F_l$ be vector bundles over $Y$. According to \cite{LSY2013} one has the following results:
			\begin{itemize}
				\item If $F$ is ample, $L$ is nef and ${\rm rank}(F)>1$, then $S^kF \otimes({\rm det} F)^2\otimes \omega_Y\otimes L$ is Nakano positive for $k \geq \max\{m-{\rm rank}(F), 0\}$ and $\omega_Y\otimes F \otimes ({\rm det} F)^k \otimes L$ is Nakano positive for $k \geq \max\{m+1-{\rm rank}(F), 2\}$.
				\item If $F$ is ample and $L$ is nef, or $F$ is nef and $L$ is ample, then $S^mF^\ast \otimes ({\rm det} F)^t \otimes L$ is Nakano positive for $t \geq {\rm rank}(F)+m-1$.
				\item If all $F_j$ are ample and $L$ is nef, or, all $F_j$ are nef and $L$ is ample, then $S^{k_1}F_1\otimes \cdots\otimes S^{k_l}F_l\otimes {\rm det}F_1\otimes \cdots\otimes {\rm det} F_l\otimes L$ is Nakano positive for any $k_1 \geq 0,\dots, k_l\geq0$.
			\end{itemize}
	According to \cite[Chap VII]{Demailly2012} one also obtain the following results:
				\begin{itemize}
				\item If $F$ is Griffiths positive of rank $r\geq 2$, then $F^\ast\otimes ({\rm det}F)^m>_m 0$ for any integer $m\geq 1$.
				\item If $0\rightarrow S\rightarrow F\rightarrow Q\rightarrow 0$ is an exact sequence of Hermitian vector bundles and $F>_m 0$, then  $S\otimes ({\rm det}Q)^m>_m 0$.
			
			\end{itemize}		
		
			Applying Theorem \ref{thm_main_proof}, we obtain Corollary \ref{cor_main}.

			\bibliographystyle{plain}
			\bibliography{directimage}

\begin{bibdiv}
\begin{biblist}

\bib{AV1965}{article}{
      author={Andreotti, Aldo},
      author={Vesentini, Edoardo},
       title={Carleman estimates for the {L}aplace-{B}eltrami equation on
  complex manifolds},
        date={1965},
        ISSN={0073-8301},
     journal={Inst. Hautes \'{E}tudes Sci. Publ. Math.},
      number={25},
       pages={81\ndash 130},
         url={http://www.numdam.org/item?id=PMIHES_1965__25__81_0},
      review={\MR{175148}},
}

\bib{Berdtsson2009}{article}{
      author={Berndtsson, Bo},
       title={Curvature of vector bundles associated to holomorphic
  fibrations},
        date={2009},
        ISSN={0003-486X},
     journal={Ann. of Math. (2)},
      volume={169},
      number={2},
       pages={531\ndash 560},
         url={https://doi.org/10.4007/annals.2009.169.531},
      review={\MR{2480611}},
}

\bib{BP2008}{article}{
      author={Berndtsson, Bo},
      author={P\u{a}un, Mihai},
       title={Bergman kernels and the pseudoeffectivity of relative canonical
  bundles},
        date={2008},
        ISSN={0012-7094},
     journal={Duke Math. J.},
      volume={145},
      number={2},
       pages={341\ndash 378},
         url={https://doi.org/10.1215/00127094-2008-054},
      review={\MR{2449950}},
}

\bib{BPW2022}{article}{
      author={Berndtsson, Bo},
      author={P\u{a}un, Mihai},
      author={Wang, Xu},
       title={Algebraic fiber spaces and curvature of higher direct images},
        date={2022},
        ISSN={1474-7480},
     journal={J. Inst. Math. Jussieu},
      volume={21},
      number={3},
       pages={973\ndash 1028},
         url={https://doi.org/10.1017/S147474802000050X},
      review={\MR{4404131}},
}

\bib{Demailly1982}{article}{
      author={Demailly, Jean-Pierre},
       title={Estimations {$L^{2}$} pour l'op\'{e}rateur {$\bar{\partial}$}
  d'un fibr\'{e} vectoriel holomorphe semi-positif au-dessus d'une
  vari\'{e}t\'{e} k\"{a}hl\'{e}rienne compl\`ete},
        date={1982},
        ISSN={0012-9593},
     journal={Ann. Sci. \'{E}cole Norm. Sup. (4)},
      volume={15},
      number={3},
       pages={457\ndash 511},
         url={http://www.numdam.org/item?id=ASENS_1982_4_15_3_457_0},
      review={\MR{690650}},
}

\bib{Demailly1988}{article}{
      author={Demailly, Jean-Pierre},
       title={Vanishing theorems for tensor powers of an ample vector bundle},
        date={1988},
        ISSN={0020-9910},
     journal={Invent. Math.},
      volume={91},
      number={1},
       pages={203\ndash 220},
         url={https://doi.org/10.1007/BF01404918},
      review={\MR{918242}},
}

\bib{Demailly2012}{article}{
      author={Demailly, Jean-Pierre},
       title={Complex analytic and differential geometry},
        date={2012},
     journal={online},
  url={http://www-fourier.ujf-grenoble.fr/~demailly/manuscripts/agbook.pdf},
}

\bib{EV1987}{incollection}{
      author={Esnault, H.},
      author={Viehweg, E.},
       title={Rev\^{e}tements cycliques. {II} (autour du th\'{e}or\`eme
  d'annulation de {J}. {K}oll\'{a}r)},
        date={1987},
   booktitle={G\'{e}om\'{e}trie alg\'{e}brique et applications, {II} ({L}a
  {R}\'{a}bida, 1984)},
      series={Travaux en Cours},
      volume={23},
   publisher={Hermann, Paris},
       pages={81\ndash 96},
      review={\MR{907924}},
}

\bib{Fujino2018}{article}{
      author={Fujino, Osamu},
       title={Koll\'{a}r-{N}adel type vanishing theorem},
        date={2018},
        ISSN={0129-2021},
     journal={Southeast Asian Bull. Math.},
      volume={42},
      number={5},
       pages={643\ndash 646},
      review={\MR{3888433}},
}

\bib{Grauert1956}{article}{
      author={Grauert, Hans},
       title={Charakterisierung der {H}olomorphiegebiete durch die
  vollst\"{a}ndige {K}\"{a}hlersche {M}etrik},
        date={1956},
        ISSN={0025-5831},
     journal={Math. Ann.},
      volume={131},
       pages={38\ndash 75},
         url={https://doi.org/10.1007/BF01354665},
      review={\MR{77651}},
}

\bib{Griffiths1969}{incollection}{
      author={Griffiths, Phillip~A.},
       title={Hermitian differential geometry, {C}hern classes, and positive
  vector bundles},
        date={1969},
   booktitle={Global {A}nalysis ({P}apers in {H}onor of {K}. {K}odaira)},
   publisher={Univ. Tokyo Press, Tokyo},
       pages={185\ndash 251},
      review={\MR{0258070}},
}

\bib{Hoffman1989}{article}{
      author={Hoffman, Jerome~William},
       title={On {K}oll\'{a}r's vanishing theorems. {I}},
        date={1989},
        ISSN={0012-7094},
     journal={Duke Math. J.},
      volume={59},
      number={2},
       pages={521\ndash 535},
         url={https://doi.org/10.1215/S0012-7094-89-05923-1},
      review={\MR{1016901}},
}

\bib{Horing2010}{article}{
      author={H\"{o}ring, Andreas},
       title={Positivity of direct image sheaves---a geometric point of view},
        date={2010},
        ISSN={0013-8584},
     journal={Enseign. Math. (2)},
      volume={56},
      number={1-2},
       pages={87\ndash 142},
         url={https://doi.org/10.4171/LEM/56-1-4},
      review={\MR{2674856}},
}

\bib{Hormander1990}{book}{
      author={H\"{o}rmander, Lars},
       title={An introduction to complex analysis in several variables},
     edition={Third},
      series={North-Holland Mathematical Library},
   publisher={North-Holland Publishing Co., Amsterdam},
        date={1990},
      volume={7},
        ISBN={0-444-88446-7},
      review={\MR{1045639}},
}

\bib{Inayama2020}{article}{
      author={Inayama, Takahiro},
       title={{$L^2$} estimates and vanishing theorems for holomorphic vector
  bundles equipped with singular hermitian metrics},
        date={2020},
        ISSN={0026-2285},
     journal={Michigan Math. J.},
      volume={69},
      number={1},
       pages={79\ndash 96},
         url={https://doi.org/10.1307/mmj/1573700740},
      review={\MR{4071346}},
}

\bib{Iwai2021}{article}{
      author={Iwai, Masataka},
       title={Nadel-{N}akano vanishing theorems of vector bundles with singular
  {H}ermitian metrics},
        date={2021},
        ISSN={0240-2963,2258-7519},
     journal={Ann. Fac. Sci. Toulouse Math. (6)},
      volume={30},
      number={1},
       pages={63\ndash 81},
         url={https://doi.org/10.5802/afst.1666},
      review={\MR{4269136}},
}

\bib{Kodaira1953}{article}{
      author={Kodaira, K.},
       title={On a differential-geometric method in the theory of analytic
  stacks},
        date={1953},
        ISSN={0027-8424},
     journal={Proc. Nat. Acad. Sci. U.S.A.},
      volume={39},
       pages={1268\ndash 1273},
         url={https://doi.org/10.1073/pnas.39.12.1268},
      review={\MR{66693}},
}

\bib{Kollar1986_1}{article}{
      author={Koll\'{a}r, J\'{a}nos},
       title={Higher direct images of dualizing sheaves. {I}},
        date={1986},
        ISSN={0003-486X},
     journal={Ann. of Math. (2)},
      volume={123},
      number={1},
       pages={11\ndash 42},
         url={https://doi.org/10.2307/1971351},
      review={\MR{825838}},
}

\bib{LN2004}{article}{
      author={Laytimi, F.},
      author={Nahm, W.},
       title={A generalization of {L}e {P}otier's vanishing theorem},
        date={2004},
        ISSN={0025-2611,1432-1785},
     journal={Manuscripta Math.},
      volume={113},
      number={2},
       pages={165\ndash 189},
         url={https://doi.org/10.1007/s00229-003-0432-y},
      review={\MR{2128545}},
}

\bib{LN2005}{article}{
      author={Laytimi, F.},
      author={Nahm, W.},
       title={On a vanishing problem of {D}emailly},
        date={2005},
        ISSN={1073-7928,1687-0247},
     journal={Int. Math. Res. Not.},
      number={47},
       pages={2877\ndash 2889},
         url={https://doi.org/10.1155/IMRN.2005.2877},
      review={\MR{2192217}},
}

\bib{LP1975}{article}{
      author={Le~Potier, J.},
       title={Annulation de la cohomolgie \`a valeurs dans un fibr\'{e}
  vectoriel holomorphe positif de rang quelconque},
        date={1975},
        ISSN={0025-5831},
     journal={Math. Ann.},
      volume={218},
      number={1},
       pages={35\ndash 53},
         url={https://doi.org/10.1007/BF01350066},
      review={\MR{385179}},
}

\bib{LSY2013}{article}{
      author={Liu, Kefeng},
      author={Sun, Xiaofeng},
      author={Yang, Xiaokui},
       title={Positivity and vanishing theorems for ample vector bundles},
        date={2013},
        ISSN={1056-3911},
     journal={J. Algebraic Geom.},
      volume={22},
      number={2},
       pages={303\ndash 331},
         url={https://doi.org/10.1090/S1056-3911-2012-00588-8},
      review={\MR{3019451}},
}

\bib{LY2015}{article}{
      author={Liu, Kefeng},
      author={Yang, Xiaokui},
       title={Effective vanishing theorems for ample and globally generated
  vector bundles},
        date={2015},
        ISSN={1019-8385},
     journal={Comm. Anal. Geom.},
      volume={23},
      number={4},
       pages={797\ndash 818},
         url={https://doi.org/10.4310/CAG.2015.v23.n4.a3},
      review={\MR{3385779}},
}

\bib{Manivel1997}{article}{
      author={Manivel, Laurent},
       title={Vanishing theorems for ample vector bundles},
        date={1997},
        ISSN={0020-9910},
     journal={Invent. Math.},
      volume={127},
      number={2},
       pages={401\ndash 416},
         url={https://doi.org/10.1007/s002220050126},
      review={\MR{1427625}},
}

\bib{Mat2016}{article}{
      author={Matsumura, Shin-ichi},
       title={A vanishing theorem of {K}oll\'{a}r-{O}hsawa type},
        date={2016},
        ISSN={0025-5831},
     journal={Math. Ann.},
      volume={366},
      number={3-4},
       pages={1451\ndash 1465},
         url={https://doi.org/10.1007/s00208-016-1371-8},
      review={\MR{3563242}},
}

\bib{Mat2022}{article}{
      author={Matsumura, Shin-ichi},
       title={Injectivity theorems with multiplier ideal sheaves for higher
  direct images under {K}\"{a}hler morphisms},
        date={2022},
        ISSN={2313-1691},
     journal={Algebr. Geom.},
      volume={9},
      number={2},
       pages={122\ndash 158},
         url={https://doi.org/10.14231/ag-2022-005},
      review={\MR{4429015}},
}

\bib{MT2007}{article}{
      author={Mourougane, Christophe},
      author={Takayama, Shigeharu},
       title={Hodge metrics and positivity of direct images},
        date={2007},
        ISSN={0075-4102},
     journal={J. Reine Angew. Math.},
      volume={606},
       pages={167\ndash 178},
         url={https://doi.org/10.1515/CRELLE.2007.039},
      review={\MR{2337646}},
}

\bib{MT2008}{article}{
      author={Mourougane, Christophe},
      author={Takayama, Shigeharu},
       title={Hodge metrics and the curvature of higher direct images},
        date={2008},
        ISSN={0012-9593},
     journal={Ann. Sci. \'{E}c. Norm. Sup\'{e}r. (4)},
      volume={41},
      number={6},
       pages={905\ndash 924},
         url={https://doi.org/10.24033/asens.2084},
      review={\MR{2504108}},
}

\bib{MT2009}{article}{
      author={Mourougane, Christophe},
      author={Takayama, Shigeharu},
       title={Extension of twisted {H}odge metrics for {K}\"{a}hler morphisms},
        date={2009},
        ISSN={0022-040X},
     journal={J. Differential Geom.},
      volume={83},
      number={1},
       pages={131\ndash 161},
         url={http://projecteuclid.org/euclid.jdg/1253804353},
      review={\MR{2545032}},
}

\bib{Nakano1955}{article}{
      author={Nakano, Shigeo},
       title={On complex analytic vector bundles},
        date={1955},
        ISSN={0025-5645},
     journal={J. Math. Soc. Japan},
      volume={7},
       pages={1\ndash 12},
         url={https://doi.org/10.2969/jmsj/00710001},
      review={\MR{73263}},
}

\bib{Naumann2021}{article}{
      author={Naumann, Philipp},
       title={Curvature of higher direct images},
        date={2021},
        ISSN={0240-2963},
     journal={Ann. Fac. Sci. Toulouse Math. (6)},
      volume={30},
      number={1},
       pages={171\ndash 201},
         url={https://doi.org/10.5802/afst.1670},
      review={\MR{4269140}},
}

\bib{Ohsawa1980}{article}{
      author={Ohsawa, Takeo},
       title={On complete {K}\"{a}hler domains with {$C^{1}$}-boundary},
        date={1980},
        ISSN={0034-5318},
     journal={Publ. Res. Inst. Math. Sci.},
      volume={16},
      number={3},
       pages={929\ndash 940},
         url={https://doi.org/10.2977/prims/1195186937},
      review={\MR{602476}},
}

\bib{Ohsawa1984}{article}{
      author={Ohsawa, Takeo},
       title={Vanishing theorems on complete {K}\"{a}hler manifolds},
        date={1984},
        ISSN={0034-5318},
     journal={Publ. Res. Inst. Math. Sci.},
      volume={20},
      number={1},
       pages={21\ndash 38},
         url={https://doi.org/10.2977/prims/1195181825},
      review={\MR{736089}},
}

\bib{Pardon_Stern1991}{article}{
      author={Pardon, William~L.},
      author={Stern, Mark~A.},
       title={{$L^2\text{--}\overline\partial$}-cohomology of complex
  projective varieties},
        date={1991},
        ISSN={0894-0347},
     journal={J. Amer. Math. Soc.},
      volume={4},
      number={3},
       pages={603\ndash 621},
         url={https://doi.org/10.2307/2939271},
      review={\MR{1102582}},
}

\bib{Ruppenthal2014}{article}{
      author={Ruppenthal, J.},
       title={{$L^2$}-theory for the {$\overline\partial$}-operator on compact
  complex spaces},
        date={2014},
        ISSN={0012-7094},
     journal={Duke Math. J.},
      volume={163},
      number={15},
       pages={2887\ndash 2934},
         url={https://doi.org/10.1215/0012794-2838545},
      review={\MR{3285860}},
}

\bib{Schumacher2012}{article}{
      author={Schumacher, Georg},
       title={Positivity of relative canonical bundles and applications},
        date={2012},
        ISSN={0020-9910},
     journal={Invent. Math.},
      volume={190},
      number={1},
       pages={1\ndash 56},
         url={https://doi.org/10.1007/s00222-012-0374-7},
      review={\MR{2969273}},
}

\bib{SZ2022}{article}{
      author={Shentu, Junchao},
      author={Zhao, Chen},
       title={$l^2$-extension of adjoint bundles and koll\'ar's conjecture},
        date={2022},
     journal={Math. Ann.},
      volume={online},
         url={https://doi.org/10.1007/s00208-022-02432-6},
}

\bib{SZ2023}{article}{
      author={Shentu, Junchao},
      author={Zhao, Chen},
       title={Mac{P}herson's conjecture via {H}\"{o}rmander's estimate},
        date={2023},
        ISSN={1073-7928,1687-0247},
     journal={Int. Math. Res. Not. IMRN},
      number={3},
       pages={2170\ndash 2187},
         url={https://doi.org/10.1093/imrn/rnab301},
      review={\MR{4565609}},
}

\bib{Takayama2023}{article}{
      author={Takayama, Shigeharu},
       title={Singular griffiths semi-positivity of higher direct images},
        date={2023},
     journal={Math. Ann.},
      volume={online},
         url={https://doi.org/10.1007/s00208-023-02632-8},
}

\bib{Takegoshi1995}{article}{
      author={Takegoshi, Kensho},
       title={Higher direct images of canonical sheaves tensorized with
  semi-positive vector bundles by proper k\"ahler morphisms},
        date={1995},
        ISSN={0025-5831},
     journal={Math. Ann.},
      volume={303},
      number={3},
       pages={389\ndash 416},
         url={https://doi.org/10.1007/BF01460997},
      review={\MR{1354997}},
}

\bib{Viehweg2001}{incollection}{
      author={Viehweg, Eckart},
       title={Positivity of direct image sheaves and applications to families
  of higher dimensional manifolds},
        date={2001},
   booktitle={School on {V}anishing {T}heorems and {E}ffective {R}esults in
  {A}lgebraic {G}eometry ({T}rieste, 2000)},
      series={ICTP Lect. Notes},
      volume={6},
   publisher={Abdus Salam Int. Cent. Theoret. Phys., Trieste},
       pages={249\ndash 284},
      review={\MR{1919460}},
}

\bib{Zucker1979}{article}{
      author={Zucker, Steven},
       title={Hodge theory with degenerating coefficients. {$L_{2}$}\
  cohomology in the {P}oincar\'e metric},
        date={1979},
        ISSN={0003-486X},
     journal={Ann. of Math. (2)},
      volume={109},
      number={3},
       pages={415\ndash 476},
         url={https://doi.org/10.2307/1971221},
      review={\MR{534758}},
}

\end{biblist}
\end{bibdiv}
			
		\end{document}